\documentclass{amsart}
\usepackage{amssymb,amsmath,amsthm,enumerate,amscd,graphicx,color,amsfonts}
\usepackage[usenames,dvipsnames,svgnames,table]{xcolor}
\usepackage{pdfsync}
 \usepackage[colorlinks=true, pdfstartview=FitV, linkcolor=blue, citecolor=blue, urlcolor=blue]{hyperref}
\newcommand{\propref}[1]{Proposition~\ref{#1}}

\newcommand{\lemref}[1]{Lemma~\ref{#1}}
\newcommand{\eqnref}[1]{~(\ref{#1})}

\newcommand{\germ}{\mathfrak}
\newtheorem{thm}{Theorem}[section]
\newtheorem{lem}[thm]{Lemma}

\theoremstyle{definition}

\newtheorem{definition}[thm]{Definition}

\newtheorem{prop}[thm]{Proposition}


\theoremstyle{rem}
\numberwithin{equation}{section}

\usepackage{amssymb}
\begin{document}
\title{Determination of the 2- cocycles for the three point Witt algebra}

\begin{abstract}  We provide formulas for computing the cocycles on a 3-point Witt algebra $Der(R)$,  using an isomorphism between two 3-point algebras $Der(R)$ and $Der ( S)$, where the cocycle is already defined.  These cocycles can  be used to construct universal central extensions and the 3-point Virasoro, which are useful for the representation theory of a 3-point current algebra. The computations determining the cocycles on $Der(R)$ involve elegant applications of the Chu-Vandermonde convolution and other identities for sums of  binomial coefficients. 
\end{abstract}
\author{Elizabeth Jurisich}
\address{Department of Mathematics,
The College of Charleston,
Charleston SC 29424}
\email{jurisiche@cofc.edu}
\author{Renato .A. Martins}
\address{Institute of Science and Tecnology,
Federal University of Sao Paulo,
Sao Jose dos Campos SP 12247014, Brazil}
\email{martins.renato@unifesp.br}

\maketitle

\section{introduction}
 
In this paper, we consider two Lie algebras of derivations of rings,  $Der (  R)$ and  $Der (  S)$, where $  S=\mathbb C[s,s^{-1},(s-1)^{-1} ]$ and $
   R=\mathbb C[t,t^{-1},u\,|\, u^2=t^2+4t]$. Note that $  R$ and $  S$ can be viewed as rings of functions on the Riemann sphere, these algebras are a type of Krichever-Novikov algebras, see \cite{MR2058804}  and \cite{MR1039524} for example. The algebras are also known as 3-point algebras, in particular 3-point Witt algebras.  These algebras are interesting examples of generalizations of the loop algebra construction of affine Lie algebras, which have an interesting representation theory and relations to mathematical physics \cite{MR2152962}, and provide an example of a family of simple Lie algebras that are infinite-dimensional.  The central extensions two algebras  under consideration in this paper play a role similar to that of the Virasoro algebra in the representation theory of current algebras for 3-point current algebras \cite{CJM}, so construction of cocycles on our 3-point Witt algebras play an important role. The result in this paper solves a problem posed in \cite{MR1109216}, one of the fundamental papers on the 3-point Witt algebra.

  We give explicit basis, and computations to construct explicit formulas for the cocycles of $Der(  R)$ by making use of an isomorphism given in \cite{MR3245847}. This isomorphism is used in \cite{CJM} to establish that the cocycles given are not co-boundaries, and correspond to the cocycles in \cite{MR3211093}. The verification of the invariance of the cocycles under the isomorphism involves extensive and interesting use of some well known combinatorial identities of products and sums of binomial coefficients.

 \section{The 3-point Witt algebra and 2-cocycles}

In this paper, all vector spaces and algebras will be over the complex numbers $\mathbb C$. 
Given a ring $R$, let $Der (R)$ denote the Lie algebra of derivations of the ring.  

For any Lie algebra $\mathfrak l$ a 2-cocycle is defined to be a bilinear function $\varphi : \mathfrak l \times \mathfrak l \rightarrow \mathbb C$ satisfying the following conditions:
\begin{enumerate}
\item $ \varphi (u,v) = - \varphi(v,u)$\\
\item $\varphi([ u,v],w) + \varphi([v,w],u) +\varphi([ w,u],v) = 0$
\end{enumerate}
for all $u,v,w \in \mathfrak l$. 

Recall, that the space of 2-cocycles determines the central extensions of the Lie algebra $\mathfrak l$ (see for example \cite{MR1269324}). The universal central extension of the 3-point Witt algebra is used in \cite{CJM} to construct representations of a 3-point current algebra, motivating us to find an explicit formula for the cocycles on the algebra.
 
 The ring of rational functions on the Riemann sphere with poles at $s= 0, 1,$ and $\infty$ is $  S:= \mathbb C[s,s^{-1},(s-1)^{-1} ]$,  the Lie algebra $Der (  S)$ is studied extensively in \cite{MR3211093}, and is one version of the 3-point Witt algebra. 

Now define
$\mathcal R := \mathbb C[t,t^{-1},u]$, and let $ a $ denote the ideal generated by $u^2-(t^2+4t)$ ( see \cite{CJM} for a more general setting), we set $  R := \mathcal R / a$.  Central extensions of this version of the 3-point Witt algebra is used in \cite{CJM} to construct free field representations of a 3-point current algebra, making an explicit formula for computing the values of the cocycles necessary. 


The element $D = (t + 2) \frac{ \partial }{\partial u} + u  \frac{ \partial }{\partial t} $ of $Der(  R)$ will also be used to denote the corresponding derivation on $  R$, which is well defined as $D \cdot a \subset a$. The following lemma is a special case of a result in \cite{MR3245847}. 

\begin{lem}  Let  $  R := R/ a=\mathbb C[t,t^{-1},u|u^2=t^2+ 4t]$  then
$$
\text{Der}(  R)=  R\left(D \right).
$$  
\end{lem}

\begin{definition}
We call $ Der (  R)$ the  3-point Witt algebra.
\end{definition}

We work in the following basis of  ${Der}( R)$:
\begin{equation}
d_n:=t^nuD,\quad   e_n=t^nD\quad \text{for} \quad D= (t+2)\frac{\partial}{\partial u}+u\frac{\partial}{\partial t} \label{eq:basis}
\end{equation}

The commutation relations of the basis elements are:
\begin{align*}
[  d_m,  d_n]
&=(n-m)\left(  d_{m+n+1}+4  d_{m+n}\right) \\
[  e_m,  e_n]&=(n-m)  d_{m+n-1} \\
[  d_m,  e_n]&=(n-m-1)  e_{m+n+1}+(4n-4m-2)  e_{m+n}.
\end{align*}

These relations can be verified by direct computation.

Define $f:R\to S$ and $\phi:S\to R$ by 
\begin{equation}
f(t)=s^{-1}(s-1)^2,\quad f(u)=s-s^{-1},\quad \phi(s)=\frac{t+2+u}{2},\quad \phi(s^{-1})=\frac{t+2-u}{2}\label{isom1}
\end{equation}
so that 
$$
\phi((s-1)^{-1})=\frac{t^{-1}u-1}{2}.
$$

It is straightforward to show that  $f:R \to S$ is an isomorphism of rings,  in fact the following is proven in \cite{MR3245847}:

\begin{lem}  
The Lie algebras  $Der(  R)$ and $Der(  S)$  are isomorphic via $t\mapsto s^{-1}(s-1)^2$, and $u\mapsto s-s^{-1}$.
\end{lem}

\subsection{Construction of the cocycles}

Recall as above  $S = \mathbb C[s,s^{-1},(s-1)^{-1}  ]$,
and set $a_1=0$ and $a_2=1$.  For $i=1,2$ the following maps are defined in  \cite{MR3211093}  
$$
\phi_i:\text{Der}(S)\times \text{Der}(S)\to \mathbb C
$$
 by specifying that $\phi_i$ is skew symmetric, has the following values on the basis $\{ s^k, (s-1)^{-l} | k \in \mathbb Z , l \in \mathbb N \}$ of $S = \mathbb C [ s, s^{-1}, (s-1)^{-1}]$, and is extended linearly: 

\begin{equation}\label{4.1}
\phi_i(s^{k+1}\partial,
(s-a_i)^{-l+1}\partial)=\begin{cases} {k+1\choose
l+1} (l^3-l),\,\,\forall\, k,l\in\mathbb Z_+,\hspace{.5cm}&\text{ if }a_i\neq 0,\\
\delta_{k,l}(l^3-l),\,\,\forall\, k,l\in\mathbb Z &\text{ if }a_i= 0, \\
\end{cases}
\end{equation}
 \begin{align*}
\phi_i((s-a_i)^{-k}\partial, (s-a_j)^{-l}\partial) 
&=\frac{(k+l+1)!}{(k-1)!(l-1)!(a_j-a_i)^{k+1}(a_i-a_j)^{l+1}},\\ 
&=\frac{(-1)^{l+1}}{(a_j-a_i)^{k+l+2}}\binom{k+l-2}{k-1}((k+l)^3-(k+l)),\\ 
&\hskip 100pt\forall \, 
k,l\in\mathbb N, j\neq i, 
\end{align*}
and $\phi_i=0$ at all other pairs of
basis elements. From this it is easy to see that
\begin{align}\label{virasoroconstraint}
\phi_i((s-a_i)^{k+1}\partial,
(s-a_i)^{-l+1}\partial)=\delta_{k,l}(l^3-l), \,\,\forall\,
k,l\in\mathbb Z.
\end{align}
which is the motivation for the formula \eqnref{4.1}.

The following result appears in \cite{MR3211093}:

\begin{prop} [Cox, Guo, Lu, and Zhao] The above defined $\phi_i$ are $2$-cocycles on $\text{Der}(S)$ which are not $2$-coboundaries for all $i=1,2$.
\end{prop}

We now use the isomorphisms \eqnref{isom1} to define an isomorphism of $\text{Der}(R)$ and $\text{Der}(S)$ which we will then use in conjunction with the above proposition to define 2-cocycles on $\text{Der}(R)$.  First we define $\Phi_f:\text{Der}(R)\to \text{Der}(S)$ by the following 
\begin{equation}
\Phi_f(u)=fuf^{-1}.
\end{equation}
Now define 2-cocycles $\bar\phi_i$, $i=1,2$ on $\text{Der}(R)$ by 
\begin{equation}\label{eq:defn}
\bar\phi_i(u,v):=\phi_i(\Phi_f(u),\Phi_f(v)).
\end{equation}

Note that \eqref{eq:defn} determines the cocycles on our algebra $Der(R)$, but we wish to obtain a formula for computing the values of the cocycles directly. This direct formula will be used in \cite{CJM} to construct central extensions of the 3-point Witt algebra and representations of the 3-point current algebra. 

\begin{lem} For $k\in\mathbb Z$, one has
\begin{align*}
\Phi_f(t^kD)=\begin{cases}
  \sum_{a=0}^{2k}\binom{2k}{a}(-1)^{a}s^{-k+a+1}\partial_s&\quad \text{if }k\geq 0 \\ 
 \left(1+\frac{2 }{s-1}+\frac{1}{(s-1)^2}\right)\partial_s&\text{ if }k=-1. \\ 
 \sum_{b=0}^{-k+1}\binom{-k+1}{b}(s-1)^{k-b+1}\partial_s &\quad \text{if }k<-1.
\end{cases}
\end{align*}
and 
\begin{align*}
\Phi_f(t^kuD)=\begin{cases} 
 \sum_{a=0}^{2k+1}\binom{2k+1}{a}(-1)^{a}s^{k-a+2}\partial_s + \sum_{b = 0}^{2k+1} \binom{2k+1}{b} (-1)^b s^{k+1-b}\partial_s &\quad \text{if }k\geq 0 \\ 
 \sum_{a=0}^{-k+1}\binom{-k+1}{a}(s-1)^{a + 2k + 1} \partial_s +\sum_{b = 0}^{-k} \binom{-k}{b} (s-1)^{2k +b +1}\partial_s
   &\quad \text{if }k<-1\\
     \left(s + 2  + \frac{2}{(s-1)}  \right)\partial_s     &\quad \text{if }k = -1.
\end{cases}
\end{align*}
\end{lem}
\begin{proof}
For $k\geq 0$ we have
\begin{align*}
\Phi_f(t^kD)(s)&=ft^kDf^{-1}(s)=f\left(t^kD\left(\frac{t+2+u}{2}\right)\right)=f\left(t^k\left(\frac{t+2+u}{2}\right)\right) \\
&=f\left(t\right)^kf\left(\frac{t+2+u}{2}\right)=s^{1-k}(s-1)^{2k} \\
&=\begin{cases}
\sum_{a=0}^{2k}\binom{2k}{a}(-1)^as^{k-a+1}& \text{ if } k\geq 0 \\
1+\frac{2}{s-1}+\frac{1}{(s-1)^2}&\text{ if }k=-1. \\
\sum_{b=0}^{-k+1}\binom{-k+1}{b}(s-1)^{k+1-b}&\text{ if }k<-1.
\end{cases}  \\
&=\begin{cases}
\sum_{a=0}^{2k}\binom{2k}{a}(-1)^as^{k-a+1}\partial_s(s)& \text{ if } k\geq 0 \\
\left(\partial_s+\frac{2\partial_s}{s-1}+\frac{1\partial_s}{(s-1)^2}\right)(s)&\text{ if }k=-1. \\
\sum_{b=0}^{-k+1}\binom{-k+1}{b}(s-1)^{k+1-b}\partial_s(s)&\text{ if }k<0.
\end{cases}
\end{align*}
\begin{align*}  
\Phi_f(t^kuD)(s)&=ft^kuDf^{-1}(s)=f\left(t^kuD\left(\frac{t+2+u}{2}\right)\right)=f\left(t^ku\left(\frac{t+2+u}{2}\right)\right) \\
&=f(t^k)f(u)f\left(\frac{t+2+u}{2}\right) =s^{1-k}(s-s^{-1})(s-1)^{2k} \\
&=s^{-k}(s^2-1)(s-1)^{2k}=s^{-k}(s+1)(s-1)^{2k+1} \\
&=s^{-k+1}(s-1)^{2k+1} +s^{-k}(s-1)^{2k+1} \\
&=\begin{cases}
  \sum_{a=0}^{2k+1}\binom{2k+1}{a}(-1)^{a}s^{k-a+2} + \sum_{b = 0}^{2k+1} \binom{2k+1}{b} (-1)^b s^{k+1-b} &\quad \text{if }2k+1\geq 0 \\ 
 \sum_{a=0}^{-k+1}\binom{-k+1}{a}(s-1)^{a + 2k + 1}  +\sum_{b = 0}^{-k} \binom{-k}{b} (s-1)^{2k +b +1}  &\quad \text{if }2k+1<0.
\end{cases}  \\
&=\begin{cases}
  \sum_{a=0}^{2k+1}\binom{2k+1}{a}(-1)^{a}s^{k-a+2} + \sum_{b = 0}^{2k+1} \binom{2k+1}{b} (-1)^b s^{k+1-b} &\quad \text{if }2k+1\geq 0 \\ 
    \left(s + 2  + \frac{2}{(s-1)}  \right)\partial_s     &\quad \text{if }k = -1\\
    \left(\frac{4}{(s-1)}  + \frac{5}{(s-1)^2} + \frac{2}{(s-1)^3 }+1  \right)\partial_s     &\quad \text{if }k = -2 \\
 \sum_{a=0}^{-k+1}\binom{-k+1}{a}(s-1)^{a + 2k + 1}  +\sum_{b = 0}^{-k} \binom{-k}{b} (s-1)^{2k +b +1}  &\quad \text{if }2k+1<-3  
\end{cases}
\end{align*}
\end{proof}

\begin{thm} \label{main} The  $\mathbb C$ valued bilinear functions $ \bar\phi_i(u,v)$  defined by \eqref{eq:defn} on $Der (R ) \times Der(R) $ are cocycles that are not co-boundaries. On the basis elements, the cocycles have the values:

 For all $k,l\in\mathbb Z$ 
\begin{align*}
\bar\phi_1(e_k, e_l) 
& =2 (l^2-l) (2 l-1)\delta_{k+l,1} +(l^3-l)\delta_{k+l,0}    \\
\bar\phi_1(e_k,d_l)=  & =6(-1)^{k+l}2^{k+l}(k-1)kl  \dfrac{(2k+2l-3)!!}{ (k+l+1)!}.\\
\bar\phi_1(d_l,e_k)&  =- \bar\phi_1(e_k,d_l)\\
\bar\phi_1(d_k,d_l)  &=l (l+1) (l+2) \delta _{k+l,-2}+4 l (2 l+1) ( l+1) \delta _{k+l,-1}
+4 l (2 l-1) (2 l+1) \delta _{k+l,0}.
\end{align*}
where, by definition, $(2k+2l-3)!!=(2k+2l-3)\cdot(2k+2l-5)\cdot ... \cdot5\cdot3\cdot1$.

For all $k,l\in\mathbb Z$

\begin{align*}
\bar\phi_2 (e_k,e_l) &=-2(l^3-l)\delta_{l+k,0}-4l(l-1)(2l-1)\delta_{k+l,1}\\
\bar\phi_2 (e_k,d_l) &= -\bar\phi_2 (d_l,e_k) = 0\\
 \bar\phi_2(d_k,d_l)=&-2l(l+1)(l+2)\delta_{k+l,-2}-8l(l+1)(2l+1)\delta_{k+l,-1}\\
&-8l(2l-1)(2l+1)\delta_{k+l,0}.\\
\end{align*}

Note that $\bar\phi_2(d_k, d_l)= -2\bar\phi_1(d_k,d_l)$, and  $\bar\phi_2(e_k,e_l)= -2\bar\phi_1(e_k,e_l)$.

\end{thm}

This theorem will be used in \cite{CJM} to define cocycles on $Der(R)$ and construct 3-point Virasoro algebras and representations of the 3-point Virasoro and Current algebras.

  \section{Proof of the theorem}

   The proofs of the values of the cocycles involve interesting combinatorial identities of binomial coefficients, and the overall approach to computing the values of the cocycles is similar in some aspects to binomial inversion, where we have additional variables, and summations.

We will prove Theorem \ref{main} by applying a variety of identities for binomial coefficients. The general approach is often to apply the subset identity if necessary to eliminate a variable, then use the Chu-Vandermonde identity, and other well known properties of binomial coefficients. 
\begin{lem}\label{tktl}
For all $k,l\in\mathbb Z$ one has
\begin{align*}
\bar\phi_1(e_k,e_l)=\bar\phi_1(t^kD,t^lD)
 &=2 (l^2-l) (2l-1)\delta_{k+l,1} +(l^3-l)\delta_{k+l,0}.
\end{align*}
\end{lem}

\begin{proof}
Note that if $k<-1, l = -1$, $l^3-l = 0$ so  $\bar\phi_1(u,v):=\phi_1(\Phi_f(u),\Phi_f(v))= 0$.
The case $ k \geq 0 , l = 1$ also results in $0$.

For $k,l\geq 0$, by \eqnref{virasoroconstraint}
\begin{align*}
&\bar\phi_1(t^kD,t^lD)=\phi_1\left( \sum_{a=0}^{2k}\binom{2k}{a}(-1)^{a}s^{-k+a+1}\partial_s,\sum_{b=0}^{2l}\binom{2l}{b}(-1)^{b}s^{-l+b+1}\partial_s\right)\\
&= \sum_{a=0}^{2k}\sum_{b=0}^{2l}\binom{2k}{a}\binom{2l}{b}(-1)^{a+b}\phi_1\left(s^{-k+a+1}\partial_s,s^{-l+b+1}\partial_s\right)\\
&= \sum_{a=0}^{2k}\sum_{b=0}^{2l}\binom{2k}{a}\binom{2l}{b}(-1)^{a+b}\delta_{k-a,-l+b}\left((-l+b)^3-(-l+b)\right)\\
&=(-1)^{k+l} \sum_{a=0}^{2k}\binom{2k}{a}\binom{2l}{k+l-a}\left((k-a)^3-(k-a)\right) \\ 
&=-(-1)^{k+l} \sum_{a=0}^{2k}\binom{2k}{a}\binom{2l}{k+l-a}(a^3-3a^2+2a)\\
&\quad+(3k-3)(-1)^{k+l} \sum_{a=0}^{2k}\binom{2k}{a}\binom{2l}{k+l-a}(a^2-a)\\ 
&\quad-3k(k-1)(-1)^{k+l} \sum_{a=0}^{2k}\binom{2k}{a}\binom{2l}{k+l-a}a\\
&\quad-(k^3-k)(-1)^{k+l} \sum_{a=0}^{2k}\binom{2k}{a}\binom{2l}{k+l-a}\\ 
&=-(-1)^{k+l} 2k(2k-1)(2k-2)\binom{2k+2l-3}{k+l}+(-1)^{k+l}3(k-1)2k(2k-1)\binom{2k+2l-2}{k+l}  \\ 
&\quad -(-1)^{k+l} 6k^2(k-1)\binom{2l+2k-1}{l+k}+(-1)^{k+l} (k^3-k)\binom{2l+2k}{l+k}  \\
&=0.
\end{align*}

For $l\geq 0$, $k=-1$ by \eqnref{virasoroconstraint}
\begin{align*}
\bar\phi_1(t^{-1}D,t^lD)&=\phi_1\left( \left(1+\frac{2 }{s-1}+\frac{1}{(s-1)^2}\right)\partial_s,\sum_{a=0}^{2l}\binom{2l}{a}(-1)^{a}s^{-l+a+1}\partial_s\right)\\
&=12\delta_{l,2}
\end{align*}
and when $k,l<-1$, then by (\ref{virasoroconstraint}):
\begin{align*}
 \bar\phi_1(t^kD,t^lD)&=\sum_{a=0}^{-l+1}\sum_{b=0}^{-k+1}\binom{-l+1}{a}\binom{-k+1}{b}\delta_{l,-k}(l^3-l)\\
&=0.
\end{align*}
   
For $k\geq 0,l<-1$, by \eqnref{virasoroconstraint}
\begin{align*}
\bar\phi_1&(t^kD,t^lD)=\phi_1\left( \sum_{a=0}^{2k}\binom{2k}{a}(-1)^{a}s^{-k+a+1}\partial_s,\sum_{b=0}^{-l+1}\binom{-l+1}{b}(s-1)^{l-b+1}\partial_s\right)\\
&= \sum_{a=0}^{2k}\sum_{b=0}^{-l+1}\binom{2k}{a}(-1)^{a}\binom{-l+1}{b}\phi_1\left(s^{-k+a+1}\partial_s,(s-1)^{l-b+1}\partial_s\right) \\
&= \sum_{a=0}^{k-2}\sum_{b=0}^{-l+1}(-1)^{a+l-b}\binom{2k}{a}\binom{-l+1}{b}\frac{(k-a-l+b-1)!}{(k-a-2)!(-l+b-2)!} \\
&=\sum_{a=0}^{k-2}\sum_{b=0}^{-l+1}(-1)^{a+l-b}\binom{2k}{a}\binom{-l+1}{b}\binom{k-a-l+b-1}{-l+b+1}((-l+b)^3-(-l+b))\\
&=\sum_{a=0}^{k-2}\sum_{b=0}^{-l+1}(-1)^{a+l-b}\binom{2k}{a}\binom{-l+1}{b}\binom{k-a-l+b-1}{-l+b+1} b(b-1)(b-2) \\
&\quad -3(l-1)\sum_{a=0}^{k-2}\sum_{b=0}^{-l+1}(-1)^{a+l-b}\binom{2k}{a}\binom{-l+1}{b}\binom{k-a-l+b-1}{-l+b+1}b(b-1)  \\
&\quad +3 \left( l^2-l\right)\sum_{a=0}^{k-2}\sum_{b=0}^{-l+1}(-1)^{a+l-b}\binom{2k}{a}\binom{-l+1}{b}\binom{k-a-l+b-1}{-l+b+1}b \\
&\quad -(l^3-l)\sum_{a=0}^{k-2}\sum_{b=0}^{-l+1}(-1)^{a+l-b}\binom{2k}{a}\binom{-l+1}{b}\binom{k-a-l+b-1}{-l+b+1} \\
&=2 (l^2-l) (2 l-1)\delta_{k+l,1} +(l^3-l)\delta_{k+l,0}.
\end{align*}

Where we have applied the identity $\binom{n}{k}=(-1)^k\binom{k-n-1}{k}$ for $n\geq0$ and also the Chu-Vandermonde convolution formula  $\sum_{k=0}^n\binom{s}{k}\binom{t}{n-k}=\binom{s+t}{n}$. For $n<0$ and $k\in\mathbb Z$ we take
$$
\binom{n}{k}=\begin{cases} (-1)^k\binom{-n+k-1}{k}&\quad \text{if } k\geq 0\\ 
(-1)^{n-k}\binom{-k-1}{n-k} &\quad\text{if }k\leq n\\
0&\quad\text{otherwise} 
\end{cases}
$$
In addition, we use for $n\geq 0$, $k\geq 0$ 
$\displaystyle{\binom{n}{k}=(-1)^k\binom{-n+k-1}{k}}$.

 Thus for $k+l\leq -2$, $2k+l-b-2=k-2-b+(k+l)\leq k-b-4$

\begin{align*}
\sum_{a=0}^{k-2}&\sum_{b=0}^{-l+1}(-1)^{a+l-b}\binom{2k}{a}\binom{-l+1}{b}\binom{k-a-l+b-1}{k-2-a} \\
&=\sum_{b=0}^{-l+1}(-1)^{k+l+b}\binom{-l+1}{b}\sum_{a=0}^{k-2}\binom{2k}{a}(-1)^{k-2-a}\binom{k-a-l+b-1}{k-2-a}\\
&=\sum_{b=0}^{-l+1}(-1)^{k+l+b}\binom{-l+1}{b}\sum_{a=0}^{k-2}\binom{2k}{a}\binom{l-b-2}{k-a-2}\\
&=\sum_{b=0}^{-l+1}(-1)^{k+l+b}\binom{-l+1}{b}\binom{2k+l-b-2}{k-2}\\
&=\sum_{b=0}^{-l+1}\binom{-l+1}{b}(-1)^{k+l-b}\binom{2k+l-b-2}{k+l-b}\\
&=\sum_{b=0\atop_{2k+l-b-2<0}}^{-l+1}\binom{-l+1}{b}(-1)^{k+l-b}\binom{2k+l-b-2}{k+l-b}\\
&\quad+\sum_{b=0\atop_{k+l-b<0\leq 2k+l-b-2}}^{-l+1}\binom{-l+1}{b}(-1)^{k+l-b}\binom{2k+l-b-2}{k+l-b}\\
&\quad+\sum_{b=0\atop_{0\leq k+l-b}}^{-l+1}\binom{-l+1}{b}(-1)^{k+l-b}\binom{2k+l-b-2}{k+l-b}\\
&=\sum_{b=0\atop_{2k+l-b-2<0}}^{-l+1}\binom{-l+1}{b}(-1)^{l-b}\binom{-k-l+b-1}{k-2}+\sum_{b=0\atop_{0\leq k+l-b}}^{-l+1}\binom{-l+1}{b}\binom{-k+1}{k+l-b}\\
&=-\sum_{b=0\atop_{2k+l-b-2<0}}^{-l+1}\binom{-l+1}{b} \binom{-k+1}{k+l-b}+\sum_{b=0\atop_{k+l-b\geq 0}}^{-l+1}\binom{-l+1}{b}\binom{-k+1}{k+l-b}\\
&=-\sum_{b= 2k+l-1}^{-l+1}\binom{-l+1}{b} \binom{-k+1}{k+l-b}+\sum_{b=0}^{\min\{k+l,-l+1\}}\binom{-l+1}{b}\binom{-k+1}{k+l-b}\\
&=(-1)^{k+l}\binom{2k+2l-3}{k+l}\delta_{k+l>1}\\
&=(-1)^{k+l}\binom{2k+2l-3}{k+l}-\delta_{k+l,1}-\delta_{k+l,0}.
\end{align*}

 Similar calculations give us
\begin{align*}
\sum_{a=0}^{k-2}&\sum_{b=0}^{-l+1}(-1)^{a+l-b}\binom{2k}{a}\binom{-l+1}{b}\binom{k-a-l+b-1}{k-2-a}b\\
& =(-l+1)\Big((-1)^{k+l+1}\binom{2k+2l-3}{k+l-1}-\delta_{k+l,1}\Big),
\end{align*}

\begin{align*}
\sum_{a=0}^{k-2}&\sum_{b=0}^{-l+1}(-1)^{a+l-b}\binom{2k}{a}\binom{-l+1}{b}\binom{k-a-l+b-1}{k-2-a}b(b-1) \\
&=(l^2-l)(-1)^{k+l}\binom{2k+2l-3}{k+l-2},
\end{align*}

\begin{align*}
\sum_{a=0}^{k-2}&\sum_{b=0}^{-l+1}(-1)^{a+l-b}\binom{2k}{a}\binom{-l+1}{b}\binom{k-a-l+b-1}{k-2-a}b(b-1)(b-2) \\
&=(l^3-l)(-1)^{k+l}\binom{2k+2l-3}{k+l-3}.
\end{align*}

As the other cases are trivial, we have the result.
\end{proof}

\begin{lem}\label{tkutlu}
For all $k,l\in\mathbb Z$ one has
\begin{align*}
\bar\phi_1&(t^kuD,t^luD)=l (l+1) (l+2) \delta _{k+l,-2}+4 l (2 l+1) ( l+1) \delta _{k+l,-1}+4 l (2 l-1) (2 l+1) \delta _{k+l,0}.
\end{align*}
\end{lem}

\begin{proof}
For $k,l\geq 0$, by \eqnref{virasoroconstraint}
\begin{align*}
\bar\phi_1(t^kuD,t^luD)&= \sum_{a=0}^{2k+1}\sum_{b=0}^{2l+1}\binom{2k+1}{a}\binom{2l+1}{b}(-1)^{a+b}\phi_1\left(s^{k-a+2}\partial_s,s^{l-b+2}\partial_s\right)\\
&\quad + \sum_{a=0}^{2k+1}\sum_{b=0}^{2l+1}\binom{2k+1}{a}\binom{2l+1}{b}(-1)^{a+b}\phi_1\left(s^{k-a+2}\partial_s,s^{l-b+1}\partial_s\right)\\
&\quad + \sum_{a=0}^{2k+1}\sum_{b=0}^{2l+1}\binom{2k+1}{a}\binom{2l+1}{b}(-1)^{a+b}\phi_1\left(s^{k-a+1}\partial_s,s^{l-b+2}\partial_s\right)\\
&\quad +\sum_{a=0}^{2k+1}\sum_{b=0}^{2l+1}\binom{2k+1}{a}\binom{2l+1}{b}(-1)^{a+b}\phi_1\left(s^{k-a+1}\partial_s,s^{l-b+1}\partial_s\right)\\ 
&= (-1)^{k+l}\sum_{a=0}^{2k+1}\binom{2k+1}{a}\binom{2l+1}{k+l-a}((k-a)^3-(k-a))\\
&\quad +(-1)^{k+l}\sum_{a=0}^{2k+1}\binom{2k+1}{a}\binom{2l+1}{k+l-a+2}((k-a+1)^3-(k-a+1))\\
&\quad -(-1)^{k+l}\sum_{a=0}^{2k+1}\binom{2k+1}{a}\binom{2l+1}{k+l-a+1}((k-a+1)^3-(k-a+1))\\
&\quad -(-1)^{k+l}\sum_{a=0}^{2k+1}\binom{2k+1}{a}\binom{2l+1}{k+l-a+1}((k-a)^3-(k-a))\\
&= \dfrac{(-1)^{k+l}6k(k-l)l}{(k+l-1)(k+l-2)}\binom{2k+2l-1}{k+l+2}-\frac{(-1)^{k+l}6kl (l-k)}{(k+l+2)(k+l+1)}\binom{2k+2l-1}{k+l}\\
&\quad -\dfrac{(-1)^{k+l}6kl}{ (k+l+1)}\binom{2k+2l-1}{k+l} -\dfrac{(-1)^{k+l} 6 kl}{(k+l+1)}\binom{2k+2l-1}{k+l}\\
&=0
\end{align*}
where using calculations similar to those applied in \lemref{tktl}, for example
\begin{align*}
(-1)^{k+l}&\sum_{a=0}^{2k+1}\binom{2k+1}{a}\binom{2l+1}{k+l-a}
=(-1)^{k+l}\binom{2k+2l+2}{k+l+2}.
\end{align*}

For $l,k<-1$ we have
\begin{align*}
\bar\phi_1(t^kuD,t^luD)&=\sum_{a=0}^{-k+1}\sum_{b=0}^{-l+1}\binom{-k+1}{a}\binom{-l+1}{b}\phi_1((s-1)^{a+2k+1}\partial_s,(s-1)^{b+2l+1}\partial_s)\\
&\quad+\sum_{a=0}^{-k}\sum_{b=0}^{-l+1}\binom{-k}{a}\binom{-l+1}{b}\phi_1((s-1)^{a+2k+1}\partial_s,(s-1)^{b+2l+1}\partial_s)\\
&\quad+\sum_{a=0}^{-k+1}\sum_{b=0}^{-l}\binom{-k+1}{a}\binom{-l}{b}\phi_1((s-1)^{a+2k+1}\partial_s,(s-1)^{b+2l+1}\partial_s)\\
&\quad+\sum_{a=0}^{-k}\sum_{b=0}^{-l}\binom{-k}{a}\binom{-l}{b}\phi_1((s-1)^{a+2k+1}\partial_s,(s-1)^{b+2l+1}\partial_s)\\
&=0.
\end{align*}

Now consider the case  $k\geq 0$ and $l<-1$.  We have
\begin{align*}
\bar\phi_1(t^kuD,t^luD)&=\sum_{a=0}^{2k+1}\sum_{b=0}^{-l+1}(-1)^{a}\binom{2k+1}{a}\binom{-l+1}{b}\phi_1(s^{k-a+2}\partial_s,(s-1)^{b+2l+1}\partial_s)\\
&\quad+\sum_{a=0}^{2k+1}\sum_{b=0}^{-l+1}(-1)^a\binom{2k+1}{a}\binom{-l+1}{b}\phi_1(s^{k-a+1}\partial_s,(s-1)^{b+2l+1}\partial_s)\\
&\quad+\sum_{a=0}^{2k+1}\sum_{b=0}^{-l}(-1)^a\binom{2k+1}{a}\binom{-l}{b}\phi_1(s^{k-a+2}\partial_s,(s-1)^{b+2l+1}\partial_s)\\
&\quad+\sum_{a=0}^{2k+1}\sum_{b=0}^{-l}(-1)^a\binom{2k+1}{a}\binom{-l}{b}\phi_1(s^{k-a+1}\partial_s,(s-1)^{b+2l+1}\partial_s)\\
\end{align*}
where
\begin{align*}
\sum_{a=0}^{2k+1}&\sum_{b=0}^{-l+1}(-1)^{a}\binom{2k+1}{a}\binom{-l+1}{b}\phi_1(s^{k-a+2}\partial_s,(s-1)^{b+2l+1}\partial_s)\\
&=-(-1)^{k+l}(-l+1)(-l)(-l-1) \binom{2k+2l+2}{k+l+2}-3(-1)^{k+l}(2l+1)(-l+1)(-l)\binom{2k+2l+1}{k+l+2}\\
&\quad - 6(-1)^{k+l}l(2l+1)(-l+1)  \binom{2k+2l}{k+l+2}  - (-1)^{k+l}(2l+1)(2l)(2l-1)\binom{2k+2l-1}{k+l+2}\\
&\quad +l(l+1)(l+2)\delta_{k+l,-2}+3 l (l+1) (2 l+1) \delta _{k+l,-1}+  (2l+1)(2l)(2l-1) \delta_{k+l,0}.
\end{align*}

Analyzing the second term: 
\begin{align*}
\sum_{a=0}^{2k+1}&\sum_{b=0}^{-l+1}(-1)^{a}\binom{2k+1}{a}\binom{-l+1}{b}\phi_1(s^{k+1-a}\partial_s,(s-1)^{b+2l+1}\partial_s)\\
&=\sum_{a=0}^{k+1}\sum_{b=0}^{-l+1}(-1)^{a}\binom{2k+1}{a}\binom{-l+1}{b}\phi_1(s^{k+1-a}\partial_s,(s-1)^{b+2l+1}\partial_s)\\
&\quad+\sum_{a=k+2}^{2k+1}\sum_{b=0}^{-l+1}(-1)^{a}\binom{2k+1}{a}\binom{-l+1}{b}\phi_1(s^{k+1-a}\partial_s,(s-1)^{b+2l+1}\partial_s)\\
&=\sum_{a=k+2}^{2k+1}\sum_{b=0}^{-l+1}(-1)^{a+b}\binom{2k+1}{a}\binom{-l+1}{b}\binom{-k-2l+a-b-1}{-k-2+a}((-b-2l)^3-(-b-2l))\\
&=\sum_{a=0}^{k-1}\sum_{b=0}^{-l+1}(-1)^{a+b+k}\binom{2k+1}{a+k+2}\binom{-l+1}{b}\binom{a-2l-b+1}{a}((-b-2l)^3-(-b-2l))\\
&= (-1)^{k+l}(-l+1)(-l)(-l-1) \binom{2k+2l+2}{k+l+1}+3 (2l+1)  (-1)^{k+l}(-l+1)(-l) \binom{2k+2l+1}{k+l+1}\\
&\quad +6 (2l+1)l(-1)^{k+l}(-l+1)\binom{2k+2l}{k+l+1} +2l(2l-1)(2l+1)(-1)^{k+l}\binom{2k+2l-1}{k+l+1}\\
&\quad+2l(l+1)(2l+1) \delta _{k+l,-1}+ (2 l+1)2 l (2 l-1) \delta _{k+l,0}.
\end{align*}
which we obtain by simplifying as in the following example
\begin{align*}
\sum_{a=0}^{k-1}&\sum_{b=0}^{-l+1}(-1)^{a+b+k}\binom{2k+1}{a+k+2}\binom{-l+1}{b}\binom{a-2l-b+1}{a}\\
&=-(-1)^{k+l}\binom{2k+2l-1}{k+l+1}\delta_{k+l\geq2}\\
&=-(-1)^{k+l}\binom{2k+2l-1}{k+l+1} +\delta_{k+l,-1}-\delta_{k+l,0},
\end{align*}

Note that the third term
\begin{align*}
\sum_{a=0}^{2k+1}\sum_{b=0}^{-l}(-1)^a\binom{2k+1}{a}\binom{-l}{b}\phi_1(s^{k-a+2}\partial_s,(s-1)^{b+2l+1}\partial_s)
\end{align*}
 and the fourth term
\begin{align*}
\sum_{a=0}^{2k+1}\sum_{b=0}^{-l}(-1)^a\binom{2k+1}{a}\binom{-l}{b}\phi_1(s^{k-a+1}\partial_s,(s-1)^{b+2l+1}\partial_s)
\end{align*}
are very similar to the first and the second one. Performing similar calculations yields
\begin{align*}
\sum_{a=0}^{2k+1}&\sum_{b=0}^{-l}(-1)^{a}\binom{2k+1}{a}\binom{-l}{b}\bar\phi_1(s^{k-a+2}\partial_s,(s-1)^{b+2l+1}\partial_s)\\
&=-(-1)^{k+l}l(l+1)(l+2)\binom{2k+2l+2}{k+l+1}  +3(2l+1)(-1)^{k+l}l(l+1)\binom{2k+2l+1}{k+l}  \\ 
&\quad +6l(2l+1)(-1)^{k+l}l(l+1)\binom{2k+2l}{k+l-1} \  +(2l+1)(2l)(2l-1) (-1)^{k+l}\binom{2k+2l-1}{k+l-2}   \\ 
&\quad -l(l+1)(2l+1)\delta_{k+l,-1}+(2l-1)(2l)(2l+1)\delta_{k+l,0}
\end{align*}
and
\begin{align*}
\sum_{a=0}^{2k+1}&\sum_{b=0}^{-l}(-1)^{a}\binom{2k+1}{a}\binom{-l}{b}\phi_1(s^{k+1-a}\partial_s,(s-1)^{b+2l+1}\partial_s)\\
&=-(-1)^{k+l}(-l)(-l-1)(-l-2)\binom{2k+2l+2}{k+l}\\
&\quad   -3(2l+1)(-1)^{k+l}(-l)(-l-1)\binom{2k+2l+1}{k+l}\\
&\quad-6(2l+1)l(-1)^{k+l}(-l)\binom{2k+2l}{k+l}\\
&\quad-2l(2l+1)(2l-1)(-1)^{k+l} \binom{2k+2l-1}{k+l}-2l(2l+1)(2l-1)\delta_{k+l,0}.
\end{align*}

Adding all of the four summands we have for $k\geq 0$ and $l<-1$ we obtain
\begin{align*}
\bar\phi_1&(t^kuD,t^luD)
=l (l+1) (l+2) \delta _{k+l,-2}+4 l (2 l+1) ( l+1) \delta _{k+l,-1}+4 l (2 l-1) (2 l+1) \delta _{k+l,0}.
\end{align*}

Finally we consider the case  $k\geq0$ and $l=-1$. We have
\begin{align*}
\bar\phi_1&(t^kuD,t^luD)=
-12\delta_{k,1}.
\end{align*}

Clearly, when $k=l=-1$ we obtain zero. For $k\leq-1$ and $l=-1$ we have
\begin{align*}
\bar\phi_1&(t^kuD,t^luD)=\\
&=\displaystyle\sum_{a=0}^{-k+1}\binom{-k+1}{a}\Big(\phi_1\left((s-1)^{a+2k+1}\partial_s,s\partial_s\right)+\phi_1\left((s-1)^{a+2k+1}\partial_s,2(s-1)^{-1}\partial_s\right)\Big)\\
&\quad+\displaystyle\sum_{a=0}^{-k}(-1)^a\binom{-k}{a}\Big(\phi_1\left((s-1)^{a+2k+1}\partial_s,s\partial_s\right)+\phi_1\left((s-1)^{a+2k+1}\partial_s,2(s-1)^{-1}\partial_s\right)\Big)\\
&=\displaystyle\sum_{a=0}^{-k+1}\binom{-k+1}{a}\Big(\phi_1\left((s-1)^{a+2k+1}\partial_s,s\partial_s\right)\Big)+\displaystyle\sum_{a=0}^{-k}\binom{-k}{a}\Big(\phi_1\left((s-1)^{a+2k+1}\partial_s,s\partial_s\right)\Big)\\
&=0.\\
\end{align*}

Combining all the above we have, for all $k,l\in\mathbb Z$,
\begin{align*}
\bar\phi_1&(t^kuD,t^luD)=l (l+1) (l+2) \delta _{k+l,-2}+4 l (2 l+1) ( l+1) \delta _{k+l,-1}
\\
&\quad+4 l (2 l-1) (2 l+1) \delta _{k+l,0}.
\end{align*}
\end{proof}

\begin{lem}
For all $k,l\in\mathbb Z$ one has
\begin{align*}
\bar\phi_1(t^kD,t^luD)=6(-1)^{k+l}2^{k+l}(k-1)kl  \dfrac{(2k+2l-3)!!}{ (k+l+1)!}.
\end{align*}
\end{lem}

\begin{proof}
Some of these calculations are similar to the calculations done in \propref{tkutlu} so we will omit some details. For $k,l\geq 0$, by \eqnref{virasoroconstraint}
\begin{align*}
\bar\phi_1&(t^kD,t^luD)= \sum_{a=0}^{2k}\sum_{a'=0}^{2l+1}\binom{2k}{a}\binom{2l+1}{a'}(-1)^{a+a'}\phi_1\left(s^{-k+a+1}\partial_s,s^{l-a'+2}\partial_s\right)\\
&\quad + \sum_{a=0}^{2k}\sum_{a'=0}^{2l+1}\binom{2k}{a}\binom{2l+1}{a'}(-1)^{a+a'}\phi_1\left(s^{-k+a+1}\partial_s,s^{l-a'+1}\partial_s\right)\\
&= -(-1)^{k+l}\sum_{a=\max\{0,k-l-1\}}^{\min\{2k,k+l\}}\binom{2k}{a}\binom{2l+1}{a-k+l+1}((a-k)^3-(a-k))\\
&\quad +(-1)^{k+l}\sum_{a=\max\{0,k-l\}}^{\min\{2k,k+l+1\}}\binom{2k}{a}\binom{2l+1}{a-k+l}((a-k)^3-(a-k)).\\
&=-(-1)^{k+l}(2k)(2k-1)(2k-2)\binom{2k+2l-2}{k+l+1}-(3-3k)(-1)^{k+l}(2k)(2k-1)\binom{2k+2l-1}{k+l+1}\\
&\quad-3k(k-1)(-1)^{k+l}(2k)\binom{2k+2l}{k+l+1}+(k^3-k)(-1)^{k+l}\binom{2k+2l+1}{k+l+1}\\
&\quad+(-1)^{k+l}(2k)(2k-1)(2k-2)\binom{2k+2l-2}{k+l}+(3-3k)(-1)^{k+l}(2k)(2k-1)\binom{2k+2l-1}{k+l}\\
&\quad+3k(k-1)(-1)^{k+l}(2k)\binom{2k+2l}{k+l}-(k^3-k)(-1)^{k+l}\binom{2k+2l+1}{k+l}\\
&=6(-1)^{k+l}2^{k+l}(k-1)kl \dfrac{(2k+2l-3)!!}{(k+l+1)!}.
\end{align*}

Note that for $l=-1$ and $k\geq0$ we also have
%
\begin{align*}
\bar\phi_1(t^kD,t^luD)&=
6(-1)^{k+l}2^{k+l}(k-1)kl \dfrac{(2k+2l-3)!!}{(k+l+1)!}.
\end{align*}

Now observe that if $l,k<-1$ then we have
\begin{align*}
\bar\phi_1(t^kD,t^luD)&=\sum_{a=0}^{-k+1}\sum_{b=0}^{-l+1}\binom{-k+1}{a}\binom{-l+1}{b}\phi_1((s-1)^{k-a+1}\partial_s,(s-1)^{b+2l+1}\partial_s)\\
&\quad+\sum_{a=0}^{-k+1}\sum_{b=0}^{-l}\binom{-k+1}{a}\binom{-l}{b}\phi_1((s-1)^{k-a+1}\partial_s,(s-1)^{b+2l+1}\partial_s)\\
&=0,
\end{align*}
for $k=-1$ and $l<-1$ we have
\begin{align*}
\bar\phi_1(t^kD,t^luD)&=\displaystyle\sum_{a=0}^{-l+1}\binom{-l+1}{a}\phi_1\Big((2(s-1)^{-1}+(s-1)^{-2})\partial_s,(s-1)^{a+2l+1}\partial_s\Big)\\
&\quad+\displaystyle\sum_{a=0}^{-l}\binom{-l}{a}\phi_1\Big((2(s-1)^{-1}+(s-1)^{-2})\partial_s,(s-1)^{a+2l+1}\partial_s\Big)\\
&=0,
\end{align*}
for $l=-1$ and $k<-1$ we have
\begin{align*}
\bar\phi_1(t^kD,t^luD)&=\displaystyle\sum_{a=0}^{-k+1}\binom{-k+1}{a}\phi_1\Big((s-1)^{k-b+1}\partial_s,(s+(s-1)^{-1})\partial_s\Big)\\
&=0
\end{align*}
and for $k=l=-1$ we have
\begin{align*}
\bar\phi_1(t^kD,t^luD)&=\phi_1\Big((2(s-1)^{-1}+(s-1)^{-2})\partial_s,(s+2(s-1)^{-1})\partial_s\Big)\\
&=0.
\end{align*}

Combining the last four cases we have, for $l,k\leq-1$,
\begin{align*}
\bar\phi_1(t^kD,t^luD)=6(-1)^{k+l}2^{k+l}(k-1)kl \dfrac{(2k+2l-3)!!}{(k+l+1)!}.
\end{align*}

Now consider $k=-1$ and $l\geq0$ where
\begin{align*}
\bar\phi_1&(t^kD,t^luD)= \sum_{a=0}^{2l+1}\binom{2l+1}{a}(-1)^{a}\phi_1\Big((2(s-1)^{-1}+(s-1)^{-2})\partial_s,s^{l-a+2}\partial_s\Big)\\
&\quad+ \sum_{a=0}^{2l+1}\binom{2l+1}{a}(-1)^{a}\phi_1\Big((2(s-1)^{-1}+(s-1)^{-2})\partial_s,s^{l+1-a}\partial_s\Big)\\
&=-2 \sum_{a=l+3}^{2l+1}\binom{2l+1}{a}(-1)^{a}\dfrac{(-l+a)!}{0!(-l+a-3)!(-1)^2}- \sum_{a=l+3}^{2l+1}\binom{2l+1}{a}(-1)^{a}\dfrac{(-l+1+a)!}{1!(-l+a-3)!(-1)^3}\\
&\quad -2 \sum_{a=l+2}^{2l+1}\binom{2l+1}{a}(-1)^{a}\dfrac{(-l+a+1)!}{0!(-l+a-2)!(-1)^2}- \sum_{a=l+2}^{2l+1}\binom{2l+1}{a}(-1)^{a}\dfrac{(-l+a+2)!}{1!(-l+a-2)!(-1)^3}\\
&=6(-1)^{k+l}2^{k+l}(k-1)kl \dfrac{(2k+2l-3)!!}{(k+l+1)!}.
\end{align*}

For $k\geq 0$ and $l<-1$ we obtain by similar calculations to the previous cases
\begin{align*}
\bar\phi_1(t^kD,t^luD)&=\sum_{a=0}^{2k}\sum_{b=0}^{-l+1}(-1)^{a}\binom{2k}{a}\binom{-l+1}{b}\phi_1(s^{a-k+1}\partial_s,(s-1)^{b+2l+1}\partial_s)\\
&\quad+\sum_{a=0}^{2k}\sum_{b=0}^{-l}(-1)^a\binom{2k}{a}\binom{-l}{b}\phi_1(s^{a-k+1}\partial_s,(s-1)^{b+2l+1}\partial_s)\\
&=\sum_{a=0}^{k-2}\sum_{b=0}^{-l+1}(-1)^{a+b}\binom{2k}{a}\binom{-l+1}{b}\binom{k-a-2l-b-1}{-b-2l+1}((-b-2l)^3-(-b-2l))\\
&\quad+\sum_{a=0}^{k-2}\sum_{b=0}^{-l}(-1)^{a+b}\binom{2k}{a}\binom{-l}{b}\binom{k-a-2l-b-1}{-b-2l+1}((-b-2l)^3-(-b-2l))\\
&=6(-1)^{k+l}2^{k+l}(k-1)kl \dfrac{(2k+2l-3)!!}{(k+l+1)!}
\end{align*}
Where we are using computations such as:
\begin{align*}
\sum_{a=0}^{k-2}&\sum_{b=0}^{-l+1}(-1)^{a+b}\binom{2k}{a}\binom{-l+1}{b}\binom{k-a-2l-b-1}{-b-2l+1}\\
&=-(-1)^{l+k}\binom{2k+2l-2}{k+l+1}-\delta_{k+l,-1}-2\delta_{k+l,0},
\end{align*}
%
and
\begin{align*}
\sum_{a=0}^{k-2}&\sum_{b=0}^{-l}(-1)^{a+b}\binom{2k}{a}\binom{-l}{b}\binom{k-a-2l-b-1}{-b-2l+1}\\
&=(-1)^{l+k}\binom{2k+2l-2}{k+l}-\delta_{k+l,0}.
\end{align*}


Finally, for $k<-1$ and $l\geq 0$, we have
\begin{align*}
\bar\phi_1&(t^kD,t^luD)=\sum_{a=0}^{-k+1}\sum_{b=0}^{2l+1}(-1)^{b}\binom{-k+1}{a}\binom{2l+1}{b}\phi_1((s-1)^{k+1-a}\partial_s,s^{l-b+2}\partial_s)\\
&\quad+\sum_{a=0}^{-k+1}\sum_{b=0}^{2l+1}(-1)^{b}\binom{-k+1}{a}\binom{2l+1}{b}\phi_1((s-1)^{k+1-a}\partial_s,s^{l+1-b}\partial_s)\\
&=-\sum_{a=0}^{-k+1}\sum_{b=l+3}^{2l+1}(-1)^{a+b-k}\binom{-k+1}{a}\binom{2l+1}{b} \binom{a-k+b-l-2}{a-k+1} (a-k)^3-(a-k))\\
&\quad-\sum_{a=0}^{-k+1}\sum_{b=l+2}^{2l+1}(-1)^{a+b-k}\binom{-k+1}{a}\binom{2l+1}{b}\binom{ a-k+ b-l -1}{a-k+1 }((a-k)^3-(a-k))\\
&=6(-1)^{k+l}2^{k+l}(k-1)kl \dfrac{(2k+2l-3)!!}{(k+l+1)!}.
\end{align*}
%

This finishes the proof.
\end{proof}

For the second cocycle we have the following proposition. 
\begin{prop}
For all $k,l\in\mathbb Z$ one has
\begin{align*}
\bar\phi_2(t^kD,t^lD)=-2\bar\phi_1(t^kD,t^lD),
\end{align*}
\begin{align*}
\bar\phi_2(t^kD,t^luD)=0,
\end{align*}
\begin{align*}
\bar\phi_2(t^kuD,t^luD)=-2\bar\phi_1(t^kuD,t^luD).
\end{align*}
\end{prop}

\begin{proof}
The calculations here are in the same spirit as the  to the calculations done for the first cocycle and, for this reason, we will do only the calculations for $k\geq0$ and $l<-1$. We have
\begin{align*}
\bar\phi_2 (t^kD,t^lD) &=   \sum_{a= 0}^{ 2k}\sum_{b=0}^{-l+1}  (-1)^a\binom{2k}{a}   \binom{-l+1}{b}  \phi_2( s^{-k +a +1 } \partial, (s-1)^{l-b+1} \partial)\\
& =  \sum_{a= 0}^{ 2k}\sum_{b=0}^{-l+1}  (-1)^a\binom{2k}{a}   \binom{-l+1}{b} \delta_{-k+a \geq 0} \binom{a-k +1 }{-l+b +1} ((-l+b)^3-(-l+b))\\
&\quad-\sum_{a= 0}^{2k}\sum_{b=0}^{-l+1}  (-1)^a\binom{2k}{a}   \binom{-l+1}{b} \delta_{-k+a< 0}\dfrac{((-l+b-1)+(k-a-1)+1)!}{(b-l-2)!(k-a-2)!(-1)^{b-l}}\\
& =  \sum_{a= k}^{ 2k}\sum_{b=0}^{-l+1}  (-1)^a\binom{2k}{a}   \binom{-l+1}{b} \binom{a-k +1 }{-l+b +1} ((-l+b)^3-(-l+b))\\
&\quad-\sum_{a= 0}^{ k-1}\sum_{b=0}^{-l+1}  (-1)^a\binom{2k}{a}   \binom{-l+1}{b} (-1)^{b+l}\binom{b-l+k-a-1 }{k-a+1} ((k-a)^3-(k-a))\\
&=(-(l^3-l)\delta_{l+k,0}+l(l-1)(2-4l)\delta_{l+k,1})\\
&\quad-(l(l+1)(l-1)\delta_{k+l,0}+l(2l-1)(2l-2)\delta_{k+l,1})\\
&=-2(l^3-l)\delta_{l+k,0}-4l(l-1)(2l-1)\delta_{k+l,1}
\end{align*}
where we use, for example,
\begin{align*}
\sum_{a= k}^{ 2k}\sum_{b=0}^{-l+1}  (-1)^a\binom{2k}{a}\binom{-l+1}{b} \binom{a-k +1 }{-l+b +1}
&=(-1)^{l+k}\binom{2k+2l-3}{l+k}
\end{align*}
and
\begin{align*}
\sum_{a= 0}^{ k-1}\sum_{b=0}^{-l+1}  (-1)^a\binom{2k}{a}   \binom{-l+1}{b} (-1)^{b+l}\binom{b-l+k-a-1 }{k-a+1}
&=-(-1)^{k+l}\binom{2k+2l}{k+l}.
\end{align*}

Now consider 
\begin{align*}
\bar\phi_2 (t^kD,t^luD) &=\sum_{a=0}^{2k}\sum_{b=0}^{-l+1}(-1)^a\binom{2k}{a}\binom{-l+1}{b}\phi_2( s^{a-k+1} \partial_s,   (s-1)^{b+2l+1} \partial_s) \\
&\quad + \sum_{a=0}^{2k}\sum_{b=0}^{-l}(-1)^a\binom{2k}{a}\binom{-l}{b}\phi_2( s^{a-k+1} \partial_s,   (s-1)^{b+2l+1} \partial_s) \\
&=-l(l+1)(l+2)\delta_{k+l,-1}-2l(l+1)(2l+1)\delta_{k+l,0}\\
&\quad+l(l+1)(l+2)\delta_{k+l,-1}+2l(l+1)(2l+1)\delta_{k+l,0}\\
&=0,
\end{align*}
and the last one is obtained by simplifying 
\begin{align*}
\bar\phi_2(t^kuD,t^luD)
&=\sum_{a=0}^{k+1}\sum_{b=0}^{-l+1}(-1)^a\binom{2k+1}{a}\binom{-l+1}{b}\binom{k+2-a}{-b-2l+1}((-b-2l)^3-(-b-2l))\\
&\quad+\sum_{a=0}^{k}\sum_{b=0}^{-l+1}(-1)^a\binom{2k+1}{a}\binom{-l+1}{b}\binom{k+1-a}{-b-2l+1}((-b-2l)^3-(-b-2l))\\
&\quad+\sum_{a=0}^{k+1}\sum_{b=0}^{-l}(-1)^a\binom{2k+1}{a}\binom{-l}{b}\binom{k+2-a}{-b-2l+1}((-b-2l)^3-(-b-2l))\\
&\quad+\sum_{a=0}^{k}\sum_{b=0}^{-l}(-1)^a\binom{2k+1}{a}\binom{-l}{b}\binom{k+1-a}{-b-2l+1}((-b-2l)^3-(-b-2l))\\
&\quad+\sum_{a=k+2}^{2k+1}\sum_{b=0}^{-l+1}(-1)^a\binom{2k+1}{a}\binom{-l+1}{b}\binom{k+2-a}{-b-2l+1}((-b-2l)^3-(-b-2l))\\
&\quad+\sum_{a=k+1}^{2k+1}\sum_{b=0}^{-l+1}(-1)^a\binom{2k+1}{a}\binom{-l+1}{b}\binom{k+1-a}{-b-2l+1}((-b-2l)^3-(-b-2l))\\
&\quad+\sum_{a=k+2}^{2k+1}\sum_{b=0}^{-l}(-1)^a\binom{2k+1}{a}\binom{-l}{b}\binom{k+2-a}{-b-2l+1}((-b-2l)^3-(-b-2l))\\
&\quad+\sum_{a=k+1}^{2k+1}\sum_{b=0}^{-l}(-1)^a\binom{2k+1}{a}\binom{-l}{b}\binom{k+1-a}{-b-2l+1}((-b-2l)^3-(-b-2l))
\end{align*}
where, for example,
\begin{align*}
\sum_{a=0}^{k}\sum_{b=0}^{-l+1}(-1)^{a}\binom{2k+1}{a}\binom{-l+1}{b}\binom{k+1-a}{-b-2l+1}
&=(-1)^{l+k+1}\binom{2k+2l-1}{l+k+1}.\\
\end{align*}
\end{proof}
 
 It is interesting to note that even though the calculations here are more complicated and involve more variables, the spirit of the calculation is reminiscent of one of the proofs of binomial inversion. 

\section{Acknowledgement}
The authors are grateful to professor Ben Cox for many interesting and fruitful discussions. The second author would like to thank the College of Charleston for hospitality during his visit in 2014. The second author was supported by FAPESP grant (2012/02459-8).

\def\cprime{$'$} \def\cprime{$'$} \def\cprime{$'$}

 \end{document}